\documentclass[11pt,reqno]{amsart}
\usepackage{hyperref}

\usepackage{pb-diagram, graphics}
\usepackage{setspace}
\usepackage{amscd,amsfonts}
\usepackage{amssymb}
\usepackage{amsmath}
\usepackage{amsthm}
\theoremstyle{plain}
\newtheorem{thm}{Theorem}[subsection]

\newtheorem{lem}[thm]{Lemma}
\newtheorem{prop}[thm]{Proposition}

\theoremstyle{definition}
\newtheorem{defn}[thm]{Definition}

\newtheorem{rem}[thm]{Remark}

\newcommand{\B}{\mathcal{B}}

\numberwithin{equation}{subsection}

\def\Z{{\mathbb Z}}

\def\N{\mathbb N}
\def\bbc{\mathbb C}
\def\:{\colon}

\newcommand{\fg}{\mathfrak{g}}

\def \fg{\mathfrak{g}}

\def\W{\widetilde{W}}
\def \fh{\mathfrak{h}}

\def\w{\widetilde{w}}
\def \fa{\mathfrak{a}}
\def \fh{\mathfrak{h}}

\def \fa{\mathfrak{a}}

\def \fn{\mathfrak{n}}

\def\C{{\mathbb C}}

\def\Z{{\mathbb Z}}
\def\bbz{{\mathbb Z}}

\def\N{{\mathbb N}}

\def\bu{\textbf{U}}
\def\f{\mathcal{F}}
\newcommand{\supp}{\operatorname{supp}}
\newcommand{\lie}[1]{\mathfrak{#1}}\def\span{\textnormal{span}}
\def\m{\mathcal{M}}
\def\bb{\textbf{B}}

\def\comp{\textnormal{comp}}

\def\ts{\tilde{S}}

\def\inv{\textnormal{Inv}}
\begin{document}
\normalsize

\title[On the Structure of a quotient of the global Weyl module for the map superalgebra $\lie{sl}(2,1)$]{On the Structure of a quotient of the global Weyl module for the map superalgebra $\lie{sl}(2,1)$}

\author{Irfan Bagci}
\address{Department of Mathematics \\
University of North Georgia \\
Oakwood, GA 30566}
\email{irfan.bagci@ung.edu}

\author{Samuel Chamberlin}
\address{Department of Mathematics\\
Park University\\ 
Parkville, MO 64152}
\email{samuel.chamberlin@park.edu}

\begin{abstract}
Let $A$ be a commutative, associative algebra with unity over $\bbc$. Using the definition of global Weyl modules for the map superalgebras given by Calixto, Lemay, and Savage, \cite{CLS}, we explicitly describe the structure of certain quotients of the global Weyl modules for the map superalgebra $\lie{sl}(2,1)\otimes A$. We also give a nice basis for these modules. This work is an extension of Theorem 3 in \cite{FL} and can naturally be extended to similar quotients of the global Weyl modules for $\lie{sl}(n,m)\otimes A$.
\end{abstract}

\maketitle

\section{Introduction}
Global Weyl modules for the loop algebras, $\fg \otimes \C[t , t^{-1}]$, of a finite-dimensional simple complex Lie algebra, $\fg$, were introduced by Chari and Pressley in \cite{CP}.  These modules are indexed by the dominant integral weights of $\fg$, and can also be defined as projective objects in the category of those $\fg \otimes \C[t , t^{-1}]$-modules whose weights are bounded by some fixed dominant integral weight of $\fg$. There are several approaches to generalizing the above objects. In \cite{FL} global Weyl modules were defined in the setting where $\C[t , t^{-1}]$ is replaced by the coordinate ring of a complex affine variety. A more general approach was taken in \cite{CFK}. There the modules for the map algebras, $\fg \otimes A$, where $A$ is a commutative, associative complex unital algebra, were studied.

Simple finite dimensional modules for $\lie{sl}(m,n)\otimes A$ were classified by  Eswara Rao   in the multiloop case in \cite{ER}. These modules were classified for $\lie{g}\otimes A$ by Savage in \cite{Sav}. The Weyl modules for for $\lie{g}\otimes A$ were defined by Calixto, Lemay, and Savage  in \cite{CLS}. In \cite{BC, BC1} the authors defined integral forms and gave integral bases for the universal enveloping algebras of $\lie{g}\otimes A$.  The aim of this paper is to use tools developed in these papers to illuminate the structure of a quotient of the global Weyl module for the map superalgebra $\mathfrak{s}l(2,1)\otimes A$.

This paper is organized as follows:  In Section 2 we fix some notation and record the properties we are going to need in the rest of the paper. Then in Section 3 we state the main result of the paper. In Section 4 we prove the main result of the paper.

\section{Preliminaries}
\subsection{}For the remainder of this work fix $\fg:=A(1,0)=\mathfrak{sl}(2,1)$. Recall that $\fg=\mathfrak{sl}(2,1)$ consists of $3\times3$ block  matrices of the form
\[\begin{bmatrix}
\begin{array}{c c |c}
a & b & c\\
d & e & f\\\hline
g & h & a+e
\end{array}
\end{bmatrix}\]
where $a,b,c,d,e,f,g,h,\in\C$.

Given a matrix $\left[a_{i,j}\right]$ define $\varepsilon_k\left(\left[a_{i,j}\right]\right):=a_{k,k}-a_{k+1,k+1}$. Define $\alpha_1:=\varepsilon_1-\varepsilon_2$ and $\alpha_2:=\varepsilon_2-\delta_1$, where $\delta_1:=\varepsilon_3$. Fix the distinguished simple root system $\Delta:=\{\alpha_1,\alpha_2\}$. Then $R_{0}^+:=\{\alpha_1\}$ and $R_{1}^+:=\{\alpha_2,(\alpha_1+\alpha_2)\}$. Fix the following Chevalley basis for $\lie{sl}(2,1)$: $h_1:=h_{\alpha_1}=e_{1,1}-e_{2,2}$ and, $h_2:=h_{\alpha_2}=e_{2,2}+e_{3,3}$. For the root vectors take $x_1:=x_{\alpha_1}=e_{1,2}$, $x_{-1}:=x_{-\alpha_1}=e_{2,1}$, $x_2:=x_{\alpha_2}=e_{2,3}$, $x_{-2}:=x_{-\alpha_2}=e_{3,2}$, and $x_3:=x_{\alpha_1+\alpha_2}=e_{1,3}$, $x_{-3}:=x_{-(\alpha_1+\alpha_2)}=e_{3,1}$. Define $h_3:=h_1+h_2$ and note that $[x_3,x_{-3}]=h_3$.

We have the following table for the bracket in $\lie{g}$.

\[\begin{tabular}{|c|c|c|c|c|c|c|c|c|}\hline
     	& $x_1$	& $x_2$ 	& $x_3$ 	& $h_1$ 	& $h_2$ 	& $x_{-1}$   & $x_{-2}$  & $x_{-3}$  \\ \hline
$x_1$	& 0    	& $x_3$ 	& 	0 	& $-2x_1$   & $x_1$ 	& $h_1$  	& 0     	& $-x_{-2}$ \\ \hline
$x_2$	& $-x_3$   & 0     	&	0  	& $x_2$ 	& 0     	& 0      	& $h_2$ 	& $x_{-1}$  \\ \hline
$x_3$	& 	0	& 0     	&   0   	& $-x_3$	& $x_3$ 	& $-x_2$ 	& $x_1$ 	& $h_3$ \\ \hline
$h_1$	& $2x_1$   & $-x_2$	& $x_3$ 	& 0     	& 0     	& $-2x_{-1}$ & $x_{-2}$  & $-x_{-3}$ \\ \hline
$h_2$	& $-x_1$   &   0   	& $-x_3$	& 0     	& 0     	& $x_{-1}$   & 0     	& $x_{-3}$  \\ \hline
$x_{-1}$ & $-h_1$   &  0    	& $x_2$ 	& $2x_{-1}$ & $-x_{-1}$ & 0      	& $-x_{-3}$ & 0     	\\ \hline
$x_{-2}$ & 0    	&  $h_2$	& $x_1$ 	& $-x_{-2}$ & 0     	& $x_{-3}$   & 0     	& 0     	\\ \hline
$x_{-3}$ & $x_{-2}$ &  $x_{-1}$ & $h_3$ & $x_{-3}$  & $-x_{-3}$ & 0      	& 0     	& 0     	\\ \hline
\end{tabular}\]

Define nilpotent sub-superalgebras $\lie n^\pm:=\bbc x_{\pm1}\oplus\bbc x_{\pm2}\oplus\bbc x_{\pm3}$ and note that $\lie g=\lie n^-\oplus \lie h\oplus \lie n^+.$ Note $\lie n^\pm=(\lie n^\pm)_0 \oplus (\lie n^\pm)_1$ where $(\lie n^\pm)_0=\mbox{span}\{x_{\pm1}\}$ and $(\lie n^\pm)_1=\mbox{span}\{x_{\pm2},x_{\pm3}\}$.

Note that $\{x_{-1},h_1,x_1\}$ is an $\mathfrak{sl}_2$-triple.

\subsection{}

Fix a commutative associative unitary algebra $A$ over $\C$. $\bb$ will denote a fixed $\C$--basis of $A$. The \emph{map superalgebra} of $\lie{g}$ is the $\Z_2$-graded vector space  $\lie{g}\otimes A$, where   $(\lie{g}\otimes A)_0 := \lie{g}_0\otimes A $ and  $(\lie{g}\otimes A)_1:= \lie{g}_1\otimes A $, with bracket given by linearly extending the bracket
$$[z\otimes a, z'\otimes b]:=[z,z']_{\lie{g}}\otimes ab,\ z,z'\in\fg,\ a,b\in A,$$
where $ [ , ]_{\lie{g}}$  denotes the super commutator bracket in $\fg$, see also \cite{BC,Sav}. $\fg$ is embedded in $\fg\otimes A$ as $\fg\otimes1$. Let $\bu(\fg\otimes A)$ denote the universal enveloping algebra of $\fg\otimes A$. Given $u\in\bu(\fg\otimes A)$ and $r\in\Z_{\geq0}$ define
$$u^{(r)}:=\frac{u^r}{r!}\textnormal{ and }\binom{u}{r}:=\frac{u(u-1)\cdots(u-r+1)}{r!}.$$
Given any Lie superalgebra $\fa$ define $T^0(\fa):=\C$, and for all $j\geq1$, define $T^j(\fa):=\fa^{\otimes j}$, $T(\fa):=\bigoplus_{j=0}^\infty T^j(\fa)$, and $T_j(\fa):=\bigoplus_{k=0}^jT^k(\fa)$. Then set $\bu_j(\fa)$ to be the image of $T_j(\fa)$ under the canonical surjection $T(\fa)\to\bu(\fa)$. For any $u\in\bu(\fa)$ \emph{define the degree of $u$} by $$\deg u:=\min_{j}\{u\in\bu_j(\fa)\}$$

\subsection{Multisets and $p(\varphi_1,\varphi_2,\xi)$}

Given any set $S$ define a \emph{multiset of elements of $S$} to be a multiplicity function $\chi:S\to\Z_{\geq0}$. Define $\f(S):=\{\chi:S\to\Z_{\geq0}:|\supp\chi|<\infty\}$. For $\chi\in\f(S)$ define $|\chi|:=\sum_{s\in S}\chi(s)$. Notice that $\f(S)$ is an abelian monoid under function addition. Define a partial order on $\f(S)$ so that for
$\psi,\chi\in\f(S)$, $\psi\leq\chi$ if $\psi(s)\leq\chi(s)$ for all $s\in S$. Define $\f_k(S):=\{\chi\in\f(S):|\chi|=k\}$ and given $\chi\in\f(S)$ define $\f(S)(\chi):=\{\psi\in\f(S):\psi\leq\chi\}$ and $\f_k(S)(\chi):=\{\psi\in\f(S)(\chi):|\psi|=k\}$. In the case $S=A$ the $S$ will be omitted from the notation. So that $\f:=\f(A)$, $\f_k:=\f_k(A)$, $\f(\chi):=\f(A)(\chi)$ and $\f_k(\chi):=\f_k(A)(\chi)$.

If $\psi\in\f(\chi)$ we define $\chi-\psi$ by standard function subtraction. Also define functions $\pi:\f-\{0\}\to A$ by
$$\pi(\psi):=\prod_{a\in A}a^{\psi(a)}$$
and extend $\pi$ to $\f$ be setting $\pi(0)=1$. Define $\m:\f\to\Z$ by
$$\m(\psi):=\frac{|\psi|!}{\prod_{a\in A}\psi(a)!}$$
For all $\psi\in\f$, $\m(\psi)\in\Z$ because if $\supp\psi=\{a_1,\ldots,a_k\}$ then $\m(\psi)$ is the multinomial coefficient
$$\binom{|\psi|}{\psi(a_1),\ldots,\psi(a_k)}$$

For $s\in S$ define $\chi_s$ to be the characteristic function of the set $\{s\}$. Then for all $\chi\in\f(S)$
$$\chi=\sum_{s\in S}\chi(s)\chi_s$$

Define $A^0=\{\varnothing\}$ and for $n\in\Z_{>0}$ view $A^n$ as the set of functions $\{1,\ldots,n\}\to A$. Also given $\xi\in A^n$ and $j\in\{1,\ldots,n\}$ define $\xi-\xi(j)=(\xi(1),\ldots,\xi(j-1),\xi(j+1),\ldots,n)\in A^{n-1}$ (if $n=1$ then $\xi-\xi(j)=\varnothing$) and $|\xi|$=n.

Given $i\in\{1,2,3\}$, $j\in\{2,3\}$ and $S\subset A$ define $X_{\pm1}:\f(S)\to\bu(\fg\otimes A)$ and $H_i:\f(S)\to\bu(\fg\otimes A)$, and $X_{\pm j}:A^n\to\bu(\fg\otimes A)$ by $X_{\pm1}(0):=1$, $H_i(0):=1$, $X_{\pm j}(\varnothing):=1$, and for $\chi,\varphi\in \f(S)-\{0\}
$ and $n>0$,
$$X_{\pm1}(\chi):=\prod_{a\in\supp\chi}\left(x_{\pm1}\otimes a\right)^{(\chi(a))}$$
$$H_i(\varphi):=\prod_{b\in\supp\varphi}\left(h_i\otimes b\right)^{(\varphi(b))}$$
$$X_{\pm j}(\xi):=\prod_{k=1}^n\left(x_{\pm j}\otimes\xi(k)\right)$$

Recursively define functions $p_1,q_1:\f^2\to\bu(x_{-1}\otimes A)\bu(h_1\otimes A)$ and $p:\f^2\times A^n\to\bu(x_{-1}\otimes A)\bu(h_1\otimes A)\Lambda(x_{-3}\otimes A)$ by
$$p_1(\varphi,\chi):=\left\{\begin{array}{cc}1,&\varphi=\chi=0\\
-\frac{1}{|\chi|}\sum_{\psi\in\f(\chi)-\{0\}}\m(\psi)\left(h_1\otimes\pi(\psi)\right)p_1(0,\chi-\psi),&\chi\neq0,\ \varphi=0\\
-\frac{1}{|\varphi|}\sum_{\substack{d\in\supp\varphi\\ \psi\in\f(\chi)}}\m(\psi)\left(x_{-1}\otimes d\pi(\psi)\right)p_1(\varphi-\chi_d,\chi-\psi)&\varphi\neq0
\end{array}\right.$$
$$q_1(\varphi,\chi):=\left\{\begin{array}{cc}0,&|\varphi|<|\chi|\\
1,&\varphi=\chi=0\\
\frac{1}{|\chi|}\sum_{\substack{c\in\supp\chi\\ d\in\supp\varphi}}(h_1\otimes cd)q_1(\varphi-\chi_d,\chi-\chi_c),&|\chi|=|\varphi|>0\\
\sum_{\phi\in\f_{|\chi|}(\varphi)}X_{-1}(\varphi-\phi)q_1(\phi,\chi),&|\varphi|>|\chi|
\end{array}\right.$$
and $p(\varphi_1,\varphi_2,\xi):=p_1(\varphi_2,\varphi_1)X_{-3}(\xi)$.

\begin{prop}\label{degp}
Let $S\subset A$, $\chi,\varphi,\psi,\phi\in\f(S)$, with $|\phi|\geq|\psi|>0$, $i,j,k,n,r\in\Z_{\geq0}$ with $r\geq|\chi|>0$, $k\leq\chi(1)$, and $\xi\in A^n$.
Then
\begin{enumerate}

\item $X_1\left(\chi\right)X_{-1}\left(r\chi_1\right)-(-1)^{r}p_1\left((r-|\chi|)\chi_1,\chi\right)\in\bu(\fg\otimes A)(x_1\otimes A)$,

\item
$X_1(\psi)X_{-1}(\phi)-(-1)^{|\psi|}q_1(\phi,\psi)\in \left(\bu(\fg\otimes A)(x_1\otimes A)+\bu_{|\phi|-|\psi|}(x_{-1}\otimes A)\bu_{|\psi|-1}(h_1\otimes A)\right)$,

\item
$p_1(\phi,\chi)-(-1)^{|\chi|+|\phi|}X_{-1}(\phi)H_1(\chi)\in\bu_{|\phi|}(x_{-1}\otimes A)\bu_{|\chi|-1}(h_1\otimes A)$,

\item
$q_1(r\chi_1,\chi)=(x_{-1})^{(r-|\chi|)}H_1(\chi)$,

\item
$q_1(\varphi+r\chi_1,\chi)-\displaystyle\binom{\varphi(1)+r-|\chi|}{\varphi(1)}(x_{-1})^{(r-|\chi|)}X_{-1}(\varphi)H_1(\chi)$
$$\in\sum_{s=r-|\chi|+1}^{\min\{r,r-|\chi|+|\varphi|-\varphi(1)\}}(x_{-1})^{(s)}\bu_{|\varphi|+r-s-|\chi|}(x_{-1}\otimes A)\bu_{|\chi|}(h_1\otimes A)$$

\item
\begin{eqnarray*}
p_1(\phi,\chi)&=&\frac{(-1)^{k}}{\binom{\chi(1)}{k}}p_1(\phi,\chi-k\chi_1)\binom{h_1-|\phi|-|\chi|+k}{k}
\end{eqnarray*}

\item
$$\binom{h_1-i}{j}X_{-3}(\xi)=X_{-3}(\xi)\binom{h_1-i-n}{j}$$

\end{enumerate}
\end{prop}
\begin{proof}
(1) and (2) follow from \cite[Lemma 5.4]{C}. (3) can be proved by induction first on $|\chi|$ with $\phi=0$ and then on $|\phi|$. (4) can be proved in the case $r=|\chi|$ by induction on $r$. Then the case $r>|\chi|$ follows. (5) can be proved using (4). (6) can be proved by induction on $k$. (7) can be shown first by induction on $n$ with $j=1$ and then by induction on $j$.
\end{proof}

\begin{rem}
Note that Proposition \ref{degp} (3) implies that $p_1(\varphi,\chi)\in\bu_{|\varphi|}(x_{-1}\otimes A)\bu_{|\chi|}(h_1\otimes A)$ and $p(\varphi_1,\varphi_2,\xi)\in\bu_{|\varphi_2|}(x_{-1}\otimes A)\bu_{|\varphi_1|}(h_1\otimes A)\Lambda_{|\xi|}(x_{-3}\otimes A)$.
\end{rem}

\section{Main Theorem}
In this section, we define the global Weyl module for map superalgebra $\lie{sl}(2,1)\otimes A$ as in \cite{CLS}.

\subsection{}

Recall that we have fixed a triangular decomposition $\fg = \fn^{-}\oplus \fh \oplus \fn^{+}$,

where  $\lie n^\pm:=\bbc x_{\pm1}\oplus\bbc x_{\pm2}\oplus\bbc x_{\pm3}$ ,   $\lie n^\pm=(\lie n^\pm)_{0} \oplus (\lie n^\pm)_{1}$ , $(\lie n^\pm)_{0}=\mbox{span}\{x_{\pm1}\}$ and $(\lie n^\pm)_{1}=\mbox{span}\{x_{\pm2},x_{\pm3}\}$, $\fh = \mbox{span}\{h_1, h_2\}, h_3= h_1+h_2$.

\begin{defn}
Define the \emph{global Weyl--module with highest weight $\lambda$}, $W_A(\lambda)$, to be the $\fg\otimes A$--module generated by a vector $w_{\lambda}$ with defining relations
$$(\lie n^+\otimes A)w_{\lambda}=0,\hskip.2in hw_{\lambda}=\lambda(h)w_{\lambda},\hskip.2in (x_{-1})^{\lambda(h_1)+1}w_\lambda=0$$
for all $h\in \fh$.
\end{defn}

In other words  $W_A(\lambda)$ is the quotient of $U(\fg\otimes A)$ by the ideal J, where $J$ is generated by 
$$\left\{ x\otimes a,\ h-\lambda(h),\  (x_{-1})^{\lambda(h_1)+1}\bigg|\ x \in \mathfrak{n}^+,\ h\in\fh,\ a\in A\right\}.$$
The quotient of the Weyl module $W_A(\lambda)$ we will study in this work, $\W_A(\lambda)$ has the additional relation $X_{-2}(\xi)w_{\lambda}=0$ where $\xi\in A^{\lambda{h_2}+1}$.

So that $\W_A(\lambda)$ is the $\fg\otimes A$--module generated by a vector $\w_{\lambda}$ with defining relations
$$(\lie n^+\otimes A)\w_{\lambda}=0,\hskip.2in h\w_{\lambda}=\lambda(h)\w_{\lambda},\hskip.2in (x_{-1})^{\lambda(h_1)+1}\w_\lambda=0,\hskip.2in X_{-2}(\xi)\w_{\lambda}=0$$
for all $h\in \fh$ and $\xi\in A^{\lambda{h_2}+1}$.

Define $\omega_k\in\lie h^*$ by $\omega_k(h_i):=\delta_{k,i}$ for all $i,k\in\{1,2\}$. Note that $\{\omega_1,\omega_2\}$ is a basis for $\lie{h}^*$.

The main result of this work is Theorem \ref{thm} in this section.

\subsection{}

Define $V\cong\C^3$ to be the natural module for $\fg$, and write the basis as $v_1$, $v_2:=x_{-1}v_1$, and $v_3:=x_{-3}v_1$. Then $V\otimes A$ is a $\fg\otimes A$ module under the action $(z\otimes a)(w\otimes b)=zw\otimes ab$.

Note that $V=V_0\oplus V_1$ where $V_0=\C v_1\oplus\C v_2$ and $V_1=\C v_3$. So $V\otimes A=(V_0\otimes A)\oplus(V_1\otimes A)$.

Following \cite{Mus}, we define an action of the symmetric group, $S_m$, on $T^m(V\otimes A)$ as follows. Given $\sigma\in S_m$ define $\inv(\sigma):=\{(j,k)\ |\ j<k,\ \sigma(j)>\sigma(k)\}$. Then given $w=(w_1,\ldots,w_m)\in T^m(V\otimes A)$ such that $w_1,\ldots, w_m$ are homogeneous in $V\otimes A$, define $\gamma(w,\sigma)\in\{\pm1\}$ by
$$\gamma(w,\sigma)=\prod_{(j,k)\in\inv(\sigma)}(-1)^{\bar{w}_{\sigma(j)}\bar{w}_{\sigma(k)}}$$
Then $S_m$ acts on $T^m(V\otimes A)$ by linearly extending the map
$$\sigma^{-1}(w_1\otimes\dots\otimes w_m)=\gamma(w,\sigma)w_{\sigma(1)}\otimes\dots\otimes w_{\sigma(m)}$$

Define a subspace $\ts^m(V\otimes A)\subset T^m(V\otimes A)$ to be the $\C$--span of the set
$$\left\{\sum_{\sigma\in S_m}\sigma^{-1}(w_1\otimes\cdots\otimes w_m)\ \bigg|\ w_1,\ldots,w_m\in V\otimes A\right\}$$
Given $u\in V\otimes A$ define $u^{\otimes m}:=\underbrace{u\otimes u\otimes\dots\otimes u}_m$ and $u^{\otimes(m)}:=\frac{1}{m!}\underbrace{u\otimes u\otimes\dots\otimes u}_m$.

Note that, for all $m\in\Z_{\geq0}$, $T^m(\bu(\fg\otimes A))$ has a braided algebra structure with factors of odd degree skew-commuting. Define the coproduct
$$\bar{\Delta}^{m}:\bu(\fg\otimes A)\to T^{m+1}\left(\bu(\fg\otimes A)\right)$$
by extending the map $\fg\otimes A\to T^{m+1}\left(\bu(\fg\otimes A)\right)$ given by
$$z\mapsto\sum_{j=0}^m1^{\otimes j}\otimes z^{\otimes(m-j)}$$
to a superalgebra homomorphism. In particular for $u,w\in\bu(\fg\otimes A)$ with $$\Delta^1(u)=\sum u_1\otimes u_2,\textnormal{ and }\Delta^1(w)=\sum w_1\otimes w_2$$
$$\Delta^1(uw)=\sum(-1)^{|u_2||w_1|}u_1w_1\otimes u_2w_2$$
Define
$$\rho:T^m(\bu(\fg\otimes A))\to\textnormal{End }T^m(V\otimes A)$$
to be the module action given by
$$\rho(u_1,\ldots,u_m)(z_1,\ldots,z_m)=(-1)^{\sum_{j=2}^m\sum_{k=1}^{j-1}|z_j||u_k|}(u_1z_1)\otimes\dots\otimes(u_mz_m)$$
Then $T^m(V\otimes A)$ is a left $\bu(\fg\otimes A)$-module via the map $\rho\circ\Delta^{m-1}$. Moreover $\ts^m(V\otimes A)$ is a submodule under this action. Thus $\ts^m(V\otimes A)$ is a left $\bu(\fg\otimes A)$-module under this action.

Define $v:=(v_1)^{\otimes m}\in\ts^m(V\otimes A)$. The goal of this work is to prove the following Theorem.
\begin{thm} \label{thm}
The assignment $\w_{m\omega_1}\mapsto v$ extends to an isomorphism $\Phi:\W_A(m\omega_1)\to\ts^m(V\otimes A)$ of  $\bu(\fg\otimes A)$--modules. Moreover the set
$$\B:=\left\{p(\varphi_1,\varphi_2,\xi)\w_{m\omega_1}\ \big|\ \varphi_1,\varphi_2\in\f(\bb),\ \xi\in\bb^n,\ |\varphi_1|+|\varphi_2|+n=m\right\}$$ is a $\C$-basis for $\W_A(m\omega_1)$.\\

\end{thm}



\section{Proof of Main Theorem}

\subsection{A Useful Basis of $\ts^m(V\otimes A)$}

In order to show that the map $\Phi$ in Theorem \ref{thm} is onto $\ts^m(V\otimes A)$ we need the following technical lemma, which gives a nice basis for $\ts^m(V\otimes A)$.

Let $(\varphi_1,\varphi_2,\xi)\in\f^2\times A^n$ be given. Then given
$\sigma\in S_{|\varphi_1|+|\varphi_2|+n}$
define $v^\sigma(\varphi_1,\varphi_2,\xi)\in T^{|\varphi_1|+|\varphi_2|+n}(V\otimes A)$ by
$$v^\sigma(\varphi_1,\varphi_2,\xi):=\sigma^{-1}\left(\bigotimes_{a\in\supp\varphi_1}
\left(v_1\otimes a\right)^{\otimes(\varphi_1(a))}\otimes
\bigotimes_{b\in\supp\varphi_2}\left(v_2\otimes b\right)^{\otimes(\varphi_2(b))}\otimes\bigotimes_{j=1}^n\left(v_3\otimes \xi(j)\right)\right)$$
Note that the definition of $v^\sigma(\varphi_1,\varphi_2,\xi)$ depends on orderings of $\supp\varphi_1$ and $\supp\varphi_2$.
Define
$$v(\varphi_1,\varphi_2,\xi):=\sum_{\sigma\in S_{|\varphi_1|+|\varphi_2|+n}}
v^\sigma(\varphi_1,\varphi_2,\xi)\in\ts^{|\varphi_1|+|\varphi_2|+n}(V\otimes A)$$
Note that the definition of $v(\varphi_1,\varphi_2,\xi)$ does \emph{not} depend on any orderings of $\supp\varphi_1$ and $\supp\varphi_2$.

\begin{lem}\label{tsbasis}
$$\mathcal{B}:=\left\{v(\varphi_1,\varphi_2,\xi)\ \Big|\ \varphi_1,\varphi_2\in\f(\bb),\ \xi\in\bb^n,\ |\varphi_1|+|\varphi_2|+n=m\right\}$$
is a basis for $\ts^m(V\otimes A)$.
\end{lem}
\begin{proof}
It is clear that $\mathcal{B}$ spans $\ts^m(V\otimes A)$. $\mathcal{B}$ is linearly independent because the set
$$\{(v_{j_1}\otimes b_1)\otimes\dots\otimes(v_{j_m}\otimes b_m)\ |\ j_1,\ldots,j_m\in\{1,2,3\},\ b_1,\ldots,b_m\in\bb\}$$
is a basis for $T^m(V\otimes A)$ and hence is linearly independent.
\end{proof}

\subsection{}

In this subsection we will study the action of $p(\varphi_1,\varphi_2,\xi)$ on $v$and in the process we will demonstrate that $\Phi$ is onto $\ts^m(V\otimes A)$. In order to understand this module action we need to study the action of $\Delta^{k-1}$ on $p_1(\varphi,\psi)$.

Given $\chi\in\f$ and $k\in\N$ define
$$\comp_k(\chi)=\left\{\psi:\{1,\ldots,k\}\to\f(\chi)\ \Bigg|\ \sum_{j=1}^k\psi(j)=\chi\right\}$$

\begin{lem}\label{Deltap}
Let $\varphi,\chi\in\f$ and $k\in\N$ be given. Then
\begin{eqnarray*}
\Delta^{k-1}(p_1(\varphi,\chi))&=&\sum_{\substack{\theta\in\comp_k(\varphi)\\ \psi\in\comp_k(\chi)}}p_1(\theta(1),\psi(1))\otimes\dots\otimes p_1(\theta(k),\psi(k))
\end{eqnarray*}
\end{lem}
\begin{proof}
A complete proof of the formula for $p_1(\varphi,\chi)$ is given in Proposition 36 in \cite{Cham}.  The case $k=1$ is trivial. The idea is to prove the case $k=2$ first by setting $\varphi=0$ and inducting on $|\chi|$ and then to induct on $|\varphi|+|\chi|$. The case $k>2$ is then proved by induction on $k$ using the identity $\Delta^l=(1^{\otimes(l-1)}\otimes\Delta^1)\circ\Delta^{l-1}$ and the $k=2$ case.
\end{proof}

\begin{lem}\label{p1v1}
For all $a\in A$ and $\varphi,\psi\in\f$ with $|\varphi|+|\psi|>1$
\begin{eqnarray*}
p_1(\varphi,\psi)(v_1\otimes a)&=&0
\end{eqnarray*}
\end{lem}
\begin{proof}
Since each term of $p_1(\varphi,\psi)$ has $x_{-1}$ appearing $|\varphi|$ times, $p_1(\varphi,\psi)(v_1\otimes a)=0$ if $|\varphi|\geq2$. If $|\psi|\geq2$ it can be shown by induction on $|\psi|$ that $p_1(0,\psi)(v_1\otimes a)=0$. We can verify by direct calculation that $p_1(\varphi,\psi)(v_1\otimes a)=0$ if $|\xi|=|\psi|=1$.
\end{proof}

\begin{lem}\label{p1v}
Let $\varphi,\chi\in\f$ be given and $k:=|\varphi|+|\chi|$ then $$p_1(\varphi,\chi)(v_1)^{k}=(-1)^{k}v(\chi,\varphi,\varnothing)$$
\end{lem}
\begin{proof}
\begin{eqnarray*}
p_1(\varphi,\chi)(v_1)^{k}&=&\rho\left(\sum_{\substack{\theta\in\comp_k(\varphi)\\ \psi\in\comp_k(\chi)}}p_1(\theta(1),\psi(1))\otimes\dots\otimes p_1(\theta(k),\psi(k))\right)(v_1)^{k}\hskip.1in\textnormal{by Lemma }\ref{Deltap}\\
&=&\sum_{\substack{\theta\in\comp_k(\varphi)\\ \psi\in\comp_k(\chi)}}p_1(\theta(1),\psi(1))(v_1)\otimes\dots\otimes p_1(\theta(k),\psi(k))(v_1)
\end{eqnarray*}
Lemma \ref{p1v1} implies that the only non-zero terms in the previous sum will have $|\theta(j)|+|\psi(j)|\leq1$ for all $j$. $p_1(\chi_a,0)(v_1)=-(v_2\otimes a)$ and $p_1(0,\chi_b)(v_1)=-(v_1\otimes b)$. $-(v_2\otimes a)$ and $-(v_1\otimes b)$ where $a\in\supp\varphi$ and $b\in\supp\chi$ are the exact factors in definition of $v(\varphi,\chi,\varnothing)$ and they occur every possible position.
\end{proof}

\begin{prop}\label{pv}
For all $(\varphi_1,\varphi_2,\xi)\in\f^2\times A^n$ with $|\varphi_1|+|\varphi_2|+n=m$
$$p(\varphi_1,\varphi_2,\xi)v=(-1)^{|\varphi_1|+|\varphi_2|}v(\varphi_1,\varphi_2,\xi)$$
\end{prop}
\begin{proof}
Let $k=|\varphi_1|+|\varphi_2|$. Then
\begin{eqnarray*}
p(\varphi_1,\varphi_2,\xi)=p_1(\varphi_2,\varphi_1)X_{-3}(\xi)v=p_1(\varphi_2,\varphi_1)v((m-n)\chi_1,0,\xi)
\end{eqnarray*}
Because $x_{-1}v_3=0$ and $h_1v_3=0$, Lemma \ref{p1v} implies that $p_1(\varphi_2,\varphi_1)v((m-n)\chi_1,0,\xi)=(-1)^kv(\varphi_1,\varphi_2,\xi)$.
\end{proof}

Define
$$C:=\span\ \left\{p(\varphi_1,\varphi_2,\xi)\w_{m\omega_1}\ \big|\ \varphi_1,\varphi_2\in\f(\bb),\ \xi\in\bb^n,\ |\varphi_1|+|\varphi_2|+n=m\right\}$$
and note that $\Phi(C)=\B$.

\begin{lem}\label{<m}
For all $\varphi_1,\varphi_2\in\f(\bb)$, and $\xi\in\bb^n$, with $|\varphi_1|+|\varphi_2|+n<m$ $p(\varphi_1,\varphi_2,\xi)\w_{m\omega_1}\in C$.
\end{lem}
\begin{proof}
Define $k:=|\varphi_1|+|\varphi_2|+n<m$. Then
\begin{eqnarray*}
p(\varphi_1+(m-k)\chi_1,\varphi_2,\xi)\w_{m\omega_1}&=&p_1(\varphi_2,\varphi_1+(m-k)\chi_1)X_{-3}(\xi)\w_{m\omega_1}\\
&=&\frac{(-1)^{m-k}}{\binom{\varphi_1(1)+m-k}{m-k}}p_1(\varphi_2,\varphi_1)\binom{h_1-|\varphi_2|-|\varphi_1|}{m-k}X_{-3}(\xi)\w_{m\omega_1}\\
&&\textnormal{by Proposition \ref{degp}(6)}\\
&=&\frac{(-1)^{m-k}}{\binom{\varphi_1(1)+m-k}{m-k}}p_1(\varphi_2,\varphi_1)X_{-3}(\xi)\binom{h_1-k}{m-k}\w_{m\omega_1}\\
&&\textnormal{by Proposition \ref{degp}(7)}\\
&=&\frac{(-1)^{m-k}}{\binom{\varphi_1(1)+m-k}{m-k}}p(\varphi_1,\varphi_2,\xi)\w_{m\omega_1}
\end{eqnarray*}
Therefore $p(\varphi_1,\varphi_2,\xi)\w_{m\omega_1}=(-1)^{m-k}\binom{\varphi_1(1)+m-k}{m-k}p(\varphi_1+(m-k)\chi_1,\varphi_2,\xi)\w_{m\omega_1}\in C$.
\end{proof}

\begin{lem}\label{firstcase}
$$C=\sum_{r+s=0}^{m}\bu_r(x_{-1}\otimes A)\bu_{m-r-s}(h_1\otimes A)\Lambda_s(x_{-3}\otimes A)\w_{m\omega_1}$$
\end{lem}
\begin{proof}
Proposition \ref{degp}(3) implies that for all $\varphi_1,\varphi_2\in\f$ and $\xi\in A^n$ with $|\varphi_1|+|\varphi_2|+n=m$
$$p(\varphi_1,\varphi_2,\xi)\w_{m\omega_1}\in\sum_{r+s=0}^{m}\bu_r(x_{-1}\otimes A)\bu_{m-r-s}(h_1\otimes A)\Lambda_s(x_{-3}\otimes A)\w_{m\omega_1}$$
This implies that $C$ is a subset of the sum.

To show that the sum is a subset of $C$, we will use induction on $k:=n+|\psi_1|+|\psi_2|$ to prove the claim that any element of the form $X_{-1}(\psi_1)H_1(\psi_2)X_{-3}(\xi)\w_{m\omega_1}$, where $\psi_1,\psi_2\in\f(\bb)$, $\xi\in\bb^n$, and $k\leq m$ is in $C$. Since for all $a\in A$,
\begin{eqnarray*}
p(\chi_a,0,\varnothing)\w_{m\omega_1}&=&-(h_1\otimes a)\w_{m\omega_1}\\
p(0,\chi_a,\varnothing)\w_{m\omega_1}&=&-(x_{-1}\otimes a)\w_{m\omega_1}\\
p(0,0,(a))\w_{m\omega_1}&=&-(x_{-3}\otimes a)\w_{m\omega_1}
\end{eqnarray*}
we see that this claim is true for $k=1$ by Lemma \ref{<m}.
\begin{eqnarray*}
p(\psi_1,\psi_2,\xi)\w_{m\omega_1}&=&p_1(\psi_2,\psi_1)X_{-3}(\xi)\w_{m\omega_1}\\
&=&(-1)^{|\psi_1|+|\psi_2|}X_{-1}(\psi_2)H_1(\psi_1)X_{-3}(\xi)\w_{m\omega_1}+u\w_{m\omega_1}
\end{eqnarray*}
where $u\in\bu_r(x_{-1}\otimes A)\bu_s(h_1\otimes A)\Lambda_n(x_{-3}\otimes A)$ for $r+s+n<k$, by Proposition \ref{degp}(3). By the induction hypothesis $u\w_{m\omega_1}\in C$. $p(\psi_1,\psi_2,\xi)\w_{m\omega_1}\in C$ by Lemma \ref{<m}. Therefore $X_{-1}(\psi_2)H_1(\psi_1)X_{-3}(\xi)\w_{m\omega_1}\in C$.
\end{proof}

\begin{lem}\label{q1winC}
Let $\varphi,\chi\in\f$ and $\xi\in A^n$ with $|\varphi|+n=m+1$. Then
$$q_1(\varphi,\chi)X_{-3}(\xi)\w_{m\omega_1}\in C$$
\end{lem}
\begin{proof}
If $|\chi|>|\varphi|$ the claim is trivially true. If $|\chi|\leq|\varphi|$ by Proposition \ref{degp}(2)
\begin{eqnarray*}
&&\hskip-.4in\left(X_1(\chi)X_{-1}(\varphi)-(-1)^{|\chi|}q_1(\varphi,\chi)\right)X_{-3}(\xi)\w_{m\omega_1}\\
&\in&\left(\bu(\fg\otimes A)\bu(x_1\otimes A)+\bu_{|\varphi|-|\chi|}(x_{-1}\otimes A)\bu_{|\chi|-1}(h_1\otimes A)\right)\Lambda_n(x_{-3}\otimes A)\w_{m\omega_1}
\end{eqnarray*}
By weight considerations $X_1(\chi)X_{-1}(\varphi)X_{-3}(\xi)\w_{m\omega_1}=0$. $\bu(\fg\otimes A)\bu(x_1\otimes A)\Lambda_n(x_{-3}\otimes A)\w_{m\omega_1}=0$ because $[x_1,x_{-3}]\w_{m\omega_1}=-x_{-2}\w_{m\omega_1}=0$ and $x_{-2}$ and $x_{-3}$ commute. So
$$q_1(\varphi,\chi)X_{-3}(\xi)\w_{m\omega_1}\in\bu_{|\varphi|-|\chi|}(x_{-1}\otimes A)\bu_{|\chi|-1}(h_1\otimes A)\Lambda_n(x_{-3}\otimes A)\w_{m\omega_1}$$
Hence $q_1(\varphi,\chi)X_{-3}(\xi)\w_{m\omega_1}\in C$ by Lemma \ref{firstcase}.
\end{proof}

\begin{lem}\label{m+1}
Let $j,l,n\in\bbz_{\geq0}$ with $n+j+l\leq m+1$. Then
$$(x_{-1})^{(m+1-j-l-n)}\bu_j(x_{-1}\otimes A)\bu_l(h_1\otimes A)\Lambda_n(x_{-3}\otimes A)\w_{m\omega_1}\subset C$$
\end{lem}
\begin{proof}
By Lemma \ref{firstcase} it suffices to show that for all $\varphi\in\f_j$, $\chi\in\f_l$ and $\xi\in A^n$
$$(x_{-1})^{(m+1-j-l-n)}X_{-1}(\varphi)H_1(\chi)X_{-3}(\xi)\w_{m\omega_1}\in C$$
This will be proved by induction on $j$. In the case $j=0$ we have by Proposition \ref{degp}(4) and Lemma \ref{q1winC}
$$(x_{-1})^{(m+1-l-n)}H_1(\chi)X_{-3}(\xi)\w_{m\omega_1}=q_1((m+1-n)\chi_1,\chi)X_{-3}(\xi)\w_{m\omega_1}\in C$$
For the induction step we have by Proposition \ref{degp}(5) and Lemma \ref{q1winC}\\
$\displaystyle\binom{\varphi(1)+m+1-j-n-l}{\varphi(1)}(x_{-1})^{(m+1-j-l-n)}X_{-1}(\varphi)H_1(\chi)X_{-3}(\xi)\w_{m\omega_1}$
$$=q_1(\varphi+(m+1-j-n)\chi_1,\chi)X_{-3}(\xi)\w_{m\omega_1}+u$$
where
$$u\in\sum_{s=m+1-n-(j-1)-l}^{\min\{m+1-n-j,m+1-n-l-\varphi(1)\}}(x_{-1})^{(s)}\bu_{m+1-n-s-l}(x_{-1}\otimes A)\bu_{l}(h_1\otimes A)X_{-3}(\xi)\w_{m\omega_1}$$
$q_1(\varphi+(m+1-j-n)\chi_1,\chi)X_{-3}(\xi)\w_{m\omega_1}\in C$ by Lemma \ref{q1winC} and $u\in C$ by the induction hypothesis. Therefore $(x_{-1})^{(m+1-j-l-n)}X_{-1}(\varphi)H_1(\chi)X_{-3}(\xi)\w_{m\omega_1}\in C$.
\end{proof}

\begin{lem}\label{basecase}
Let $k\in\Z_{\geq0}$ and $\xi\in A^n$ with $k+n>m$. Then $\bu_k(h_1\otimes A)\Lambda_n(x_{-3}\otimes A)\w_{m\omega_1}\subset C$.
\end{lem}
\begin{proof}
The proof proceeds by induction on $k$ using Proposition \ref{degp}(1) and
(3).  If $k=0$ then $X_{-3}(\xi)\w_{m\omega_1}=0\in C$. Let $\chi\in\f_k$. Then
\begin{eqnarray*}
0&=&X_1(\chi)X_{-1}(k\chi_1)X_{-3}(\xi)\w_{m\omega_1}\hskip.2in\textnormal{by weight considerations}\\
&=&(-1)^{k}p_1(0,\chi)X_{-3}(\xi)\w_{m\omega_1}\hskip.2in\textnormal{by Proposition \ref{degp}(1)}\\
&=&H_1(\chi)X_{-3}(\xi)\w_{m\omega_1}+u\hskip.2in\textnormal{by Proposition \ref{degp}(3)}
\end{eqnarray*}
where $u\in\bu_{k-1}(h_1\otimes A)\Lambda_n(x_{-3}\otimes A)\w_{m\omega_1}\subset C$ by the induction hypothesis.
\end{proof}

\subsection{}

\noindent\emph{Proof of Theorem \ref{thm}.} The relations $(\lie n^+\otimes A)v=0$, and $hv=(m,0)(h)v$ hold for all $h\in\lie{h}$. Also it can be easily shown that $x_{-i}^{(m,0)(i)+1}v=0$, for all $i\in\{1,2\}$. Thus the map $\Phi$ in the statement of Theorem \ref{thm} gives a well-defined surjection $\W_A(m\omega_1)\to\bu(\fg \otimes A)v$. Let $v(\varphi_1,\varphi_2,\xi)\in\B$ be given. Then, by Proposition \ref{pv},
$$\Phi\left(p(\varphi_1,\varphi_2,\xi)\w_{m\omega_1}\right)=p(\varphi_1,\varphi_2,\xi)v=(-1)^{|\varphi_1|+|\varphi_2|}v(\varphi_1,\varphi_2,\xi)$$
Therefore $\Phi$ is onto $\ts^m(V\otimes A)$. To show that the map is injective we will show that a spanning set is sent to a basis. Lemma \ref{tsbasis} showed that $\B$ is a basis for $\ts^m(V\otimes A)$. Thus it will suffice to show that the set
$$\{p(\varphi_1,\varphi_2, \xi)\w_{m\omega_1}\ |\ \varphi_1,\varphi_2\in\f(\bb),\  \xi \in \bold{B}^n, \ |\varphi_1|+|\varphi_2|+n=m\}$$
spans $\W_A(m\omega_1)$ (i.e. $C=\W_A(m\omega_1)$).

By the Poincar\'{e}-Birkhoff-Witt Theorem for Lie superalgebras
$$\bu(\fg\otimes A)\cong\bu(x_{-1}\otimes A)\bu(h_1\otimes A)\Lambda(x_{-3}\otimes A) \Lambda(x_{-2}\otimes A)\bu(h_2\otimes A)\bu(\lie{n}^+\otimes A)$$ as vector spaces.
Therefore
$$\W_A(m\omega_1)\cong\bu(x_{-1}\otimes A)\bu(h_1\otimes A)\Lambda(x_{-3}\otimes A)\w_{m\omega_1},$$
because $x_{-2}\otimes a$, $h_2\otimes a$ and $\lie{n}^+\otimes A$ act by zero on $\w_{m\omega_1}$.
For injectivity it suffices to show that

$$\bu(x_{-1}\otimes A)\bu(h_1\otimes A)\Lambda(x_{-3}\otimes A)\w_{m\omega_1}\subset C$$

In fact, by weight considerations, it suffices to show that for all $ j, k,n  \in \mathbb{Z}_{\geq 0}$ with $j+n\leq m,$

$$\bu_j(x_{-1}\otimes A)\bu_k(h_1\otimes A)\Lambda_n(x_{-3}\otimes A)\w_{m\omega_1}\subset C.$$ The case $j+k+n \leq m$ follows from Lemma \ref{firstcase}. The case $j+k+n > m$ follows by induction on $j$. The $j=0$ case follows from Lemma \ref{basecase}. The $j+1$ case follows from Lemma \ref{m+1}.


Thus $C$ spans $\W_A(m\omega_1) $ and hence the map in the theorem is injective and therefore is an isomorphism of left $\bu(\fg \otimes A)$-modules.

\hskip4.5in $\square$


\begin{thebibliography}{9999}\frenchspacing

\bibitem{BC} I. Bagci and S. Chamberlin, \emph{Integral bases for the universal enveloping algebras of map superalgebras}, Journal of Pure and Applied Algebra  \textbf{218}, no. {8}, (2014),  pp. 1563-1576.

\bibitem{BC1} I. Bagci and S. Chamberlin, \emph{Integral bases for the universal enveloping algebras of map superalgebras II}, math.RT/arXiv:1308.2849v1.
    
\bibitem{CLS} Lucas Calixto, Joel Lemay and Alistair Savage, \emph{Weyl modules for Lie superalgebras}, math.RT/arXiv:1505.06949v1.

\bibitem{Cham} S. Chamberlin, \emph{Integral bases for the universal enveloping algebras of map algebras}, Doctoral dissertation accepted by the University of California, Riverside. \emph{Proquest LLC., Ann Arbor}, 2011, viii+85 pp.

\bibitem{C} S. Chamberlin, \emph{Integral bases for the universal enveloping algebras of map algebras},  J. Algebra 377 (2013), no. 1, 232--249. MR3008904

\bibitem{CFK} V. Chari, G. Fourier, and T. Khandai, \emph{A categorical approach to Weyl modules}, Transform. Groups 15 (2010), no. 3, 517--549.

\bibitem{CP} V. Chari, A. Pressley \emph{Weyl modules for classical and quantum affine algebras}, Represent. Theory 5:191--223 (electronic), 2001. MR1850556

\bibitem{ER} S. Eswara Rao,  \emph{Finite dimensional modules for multiloop superalgebra of type A(m, n) and C(m)}, Proceedings of the American Mathematical Society 141 (2013), 3411-3419

\bibitem{FL} B. Feigin and S. Loktev, \emph{Multi-dimensional Weyl modules
and symmetric functions}, Comm. Math. Phys. 251 (2004), 427--445. MR2102326 (2005m:17005)

\bibitem{Mus} I. M. Musson, \emph{Lie superalgebras and enveloping algebras}, Graduate Studies in Mathematics, 131. \emph{American Mathematical Society, Providence, RI}, 2012. xx+488 pp. ISBN: 978-08218-6867-6.  MR2906817

\bibitem{Sav}  A. Savage, \emph{Equivariant map superalgebras}, Mathematische Zeitschrift, 277 (2014), no. 1, pp. 373–399
\end{thebibliography}
\end{document}